\documentclass[amsfonts,12pt]{article}
\def \bdf{\begin{definition}}
\def \edf{\end{definition}}
\usepackage{latexsym,epsfig,graphpap,graphics,amsmath}
 \def \ad{\end{document}}
\usepackage{amssymb}
\usepackage{amsmath}
\newcounter{corollary}
\newcounter{proposition}
\newcounter{definition}
\newcounter{remark}
\newcounter{example}
\newcounter{lemma}
\newcounter{result}
\newcounter{theorem}
\usepackage{graphicx}
\usepackage{amsmath}

\newtheorem{theorem}[theorem]{Theorem}

 \newtheorem{result}[theorem]{Result}

\newtheorem{corollary}[theorem]{Corollary}

\newtheorem{definition}[theorem]{Definition}

\newtheorem{lemma}[theorem]{Lemma}

\newtheorem{proposition}[theorem]{Proposition}
\newtheorem{remark}[theorem]{Remark}

\newenvironment{proof}[1][Proof]{\noindent\textbf{#1.} }{\
\rule{0.5em}{0.5em}} \textheight=8.85in \textwidth=15.5cm
\usepackage{amssymb}
\usepackage{amssymb}
\usepackage{amsmath}
\usepackage{graphicx}
\usepackage{amsmath}

\usepackage{amssymb}

\def \n1{\newpage}

\def \no1{\noindent}

\def \fr{\frac}

\def \bprf{\begin{proof}{}}
\def \eprf{\end{proof}{}}
\def \blem{\begin{lemma}}
\def \bth{\begin{theorem}}
\def \brs{\begin{result}\rm}
\def \ers{\end{result}}
\def \eth{\end{theorem}}
\def \brem{\begin{remark}}
\def \erem{\end{remark}}
\def \bpr{\begin{proposition}}
\def \epr{\end{proposition}}
\def \bcor{\begin{corollary}}
\def \ecor{\end{corollary}}
\def \elem{\end{lemma}}
\def \1{\begin{eqnarray}}
\def \2{\end{eqnarray}}
\def \3{\begin{eqnarray*}}
\def \4{\end{eqnarray*}}
\def \6{\vspace*{7mm}}
\def \bqn{\begin{equation}}
\def \eqn{\end{equation}}

\def \s1{\subseteq}

\def \l{\left}
\def \r{\right}

\def \bt{\begin{tabular}}
\def \et{\end{tabular}}

\def \l{\left}

\def \r{\right}
\def \hs1{\hspace*{3mm}}
 \def \s1{\sqrt}
\def \q2{\hspace*{9mm}}

\def \un1{\underline}

\def \vs1{\vspace*{4mm}}

\def \ba{\begin{array}}
\def \ea{\end{array}}

\newcommand{\ec}{\end{center}}
\newcommand{\bc}{\begin{center}}
\newcommand{\be}{\begin{equation}}

\newcommand{\ee}{\end{equation}}
\newcommand{\bn}{\begin{enumerate}}
\newcommand{\en}{\end{enumerate}}
\newcommand{\bi}{\begin{itemize}}
\newcommand{\ei}{\end{itemize}}
\begin{document}

\begin{center}
{\bf LATTICE PATHS AND ORDER-PRESERVING
PARTIAL TRANSFORMATIONS}\\[4mm]
\textbf{A. Laradji}\\
${}$\\
Department of Mathematics and Statistics\\
King Fahd University of Petroleum \& Minerals \\
Dhahran 31261\\
Saudi Arabia\\
\vspace{0.2cm}
and\\
\vspace{0.2cm}
\textbf{A. Umar}\\
${}$\\
Department of Mathematics and Statistics\\
Sultan Qaboos University \\
Al-Khod PC 123\\
Sultanate of Oman\\

email: \texttt{alaradji\mbox{@}kfupm.edu.sa} and
\texttt{aumarh\mbox{@}squ.edu.om}
\end{center}

\begin{abstract}
Let ${\cal PO}_n$ be the semigroup of all order-preserving partial
transformations of a finite chain.  It is shown that there exist
bijections between the set of certain lattice paths in the
Cartesian plane that start at $(0,0)$, end at $(n-1,n-1)$, and
certain subsemigroups of ${\cal PO}_n$. Several consequences of
these bijections were discussed.
\end{abstract}
\noindent KEYWORDS: SEMIGROUP, SUBSEMIGROUP, PARTIAL
ORDER-PRESERVING TRANSFORMATION, PARTIAL ORDER-DECREASING
TRANSFORMATION, LATTICE PATH, SCHR\"{O}DER NUMBER, CATALAN NUMBER,
CENTRAL DELANNOY NUMBER.

MSC2010: 20M18, 20M20, 05A10, 05A15.
\baselineskip=22pt

\section{Introduction}
\setcounter{equation}{0} Consider a finite chain, say $X_n =
\{1,2,\ldots,n\}$ under the natural ordering and let $T_n$ and
$P_n$ be the full transformation semigroup and the semigroup of
all partial transformations on $X_n$, under the usual composition,
respectively.  We shall call a partial transformation $\alpha: X_n
\rightarrow X_n$, {\em order-decreasing (order-increasing)} or
simply {\em decreasing (increasing)} if $x\alpha \leq x \,(x\alpha
\geq x)$ for all $x$ in $\mbox{Dom }\alpha$, and $\alpha$ is {\em
order-preserving} if $x\leq y$ implies $x\alpha \leq y\alpha$ for
$x,y$ in $\mbox{Dom }\alpha$. This paper establishes bijections
between certain types of lattice paths and certain subsemigroups
of  ${\cal PO}_n$, the semigroup of all order-preserving partial
transformations of $X_n$. Thus combinatorial problems of these
lattice paths can be translated into combinatorial problems of
these subsemigroups of  ${\cal PO}_n$, and vice-versa.\\
\noindent Various enumerative problems of an essentially
combinatorial nature have been considered for certain classes of
semigroups of transformations. For example, it is well known and
indeed obvious that $T_n$ and $P_n$ have orders $n^n$ and
$(n+1)^n$, respectively.  Only slightly less obvious are their
number of idempotents given by
\[
|E(T_n)| = \sum^n_{r=1} \left(\begin{array}{c} n\\ r\end{array}
\right) r^{n-r} \mbox{ and } |E(P_n)| = \sum^{n+1}_{r=1}
\left(\begin{array}{c} n \\ r-1\end{array} \right)r^{n+1-r}.\] The
first usually attributed to Tainiter \cite{Tai} is actually Ex
2.2.2(a) in Clifford and Preston \cite{Cli}.  The second can be
deduced easily via Vagner's method of representing a partial
transformation by a full transformation \cite{Vag}, which has been
used to good effect by Garba \cite{Gar1}.  The following list
(which is by no means exhaustive) of papers and books
\cite{Gar1,Gar2,Gom,Hig1,Hig2,How1,How2,How3,Lar1,Lar2,Lar3,
Lar4,Lar5,Umar1,Umar2} each contains some interesting
combinatorial results pertaining to semigroups of transformations.
Combinatorial properties of ${\cal PC}_n$, the semigroup of all
decreasing and order-preserving partial transformations of $X_n$
have been investigated recently by Laradji and Umar \cite{Lar2}.
One surprising outcome from \cite{Lar2, Lar3} is the discovery of
many integer sequences that are not yet recorded or have just
recently been recorded in Sloane's encyclopaedia of integer
sequences \cite{Slo}. The advantage of the approach in this paper
is that it establishes a connection
between the theories of lattice paths and partial transformations.\\
\noindent In Section 2, we give the necessary definitions that we
need in the paper as well as show that the semigroup ${\cal PC}_n$
of all decreasing and order-preserving partial transformations of
$X_n$, is a disjoint union of two subsemigroups of the same
cardinality. In Section 3, we establish a bijection between a set
of certain lattice paths and ${\cal PC}_n$. Hence we obtain the
order of ${\cal PC}_n$ as the large or double Schr\"{o}der number
\cite{Per,Sta}, and the order of the two subsemigroups as the
small Schr\"{o}der number. We also discuss several consequences of
this bijection. In Section 4, we gather some remarks concerning
the results of the paper and we also give an alternative proof of
the result that set of all idempotents of ${\cal PC}_n$ is of
cardinality $(3^n+1)/2$.

\section{Preliminaries}
\setcounter{equation}{0} \setcounter{lemma}{0}
\setcounter{theorem}{0} \setcounter{corollary}{0}
\setcounter{proposition}{0} \setcounter{remark}{0} For standard
terms and concepts in transformation semigroup theory see
\cite{How1} or \cite{Hig1}.  We now recall some definitions and
notations to be used in the paper.  Consider $X_n =
\{1,2,\ldots,n\}$ and let $\alpha: X_n\rightarrow X_n$ be a
partial transformation. We shall denote by $\mbox{Dom }\alpha$ and
$\mbox{Im }\alpha$, the {\em domain} and {\em image set} of
$\alpha$, respectively. The semigroup $P_n$, of all partial
transformations contains two important subsemigroups which have
been studied recently.  They are $PD_n$ and ${\cal PO}_n$ the
semigroups of all order-decreasing and respectively
order-preserving partial transformations, (see \cite{Umar3} and
\cite{Gar2,Gom,Lar3}). Now let
\begin{equation}
{\cal PC}_n = PD_n \cap {\cal PO}_n
\end{equation}
be the semigroup of all decreasing and order-preserving partial
transformations of $X_n$.  Next let
\begin{equation}
Q_n = \{\alpha \in {\cal PC}_n: 1\in \mbox{ Dom }\alpha\}
\end{equation}
be the set of all maps in ${\cal PC}_n$ all of whose domain does
contain the element 1.  Then evidently we have the following
result.
\begin{lemma}
Both $Q_n$ and its set complement $Q'_n$ are subsemigroups of
${\cal PC}_n$. Moreover, $Q_n\cdot Q'_n = Q'_n\cdot Q_n = Q'_n$.
\end{lemma} Less evidently, we have
\begin{lemma}
$|Q_n| = |Q'_n|$. \end{lemma}
\begin{proof}{}
Define a map $\phi$ from $Q_n$ into $Q'_n$ by
\[
\phi(\alpha) = \alpha' \quad (\alpha\in Q_n, \alpha' \in Q'_n)
\]
where
\[
x\alpha' = x\alpha \quad (\mbox{for all } x\in \mbox{Dom
}\alpha\setminus \{1\}).\] It is clear that $\phi$ is a bijection
since $1\not\in \mbox{Dom }\alpha'$ for all $\alpha'$ in $Q'_n$,
and $1\alpha = 1$ for all $\alpha$ in $Q_n$. Informally, we may
argue as follows: to get $Q'_n$ from $Q_n$ we delete 1 from
$\mbox{Dom }\alpha$ (for each $\alpha \in Q_n$), if
$|1\alpha^{-1}| \geq 2$; otherwise we delete 1 from $\mbox{Im
}\alpha$ as well. Conversely, to get $Q_n$ from $Q'_n$ we extend
each $\alpha' \in Q'_n$ by defining $1\alpha = 1$ and $x\alpha = x
\alpha'$ for all $x\geq 2$.
\end{proof}

\section{The order of ${\cal PC}_n$}
\setcounter{equation}{0} \setcounter{lemma}{0}
\setcounter{theorem}{0} \setcounter{corollary}{0}
\setcounter{proposition}{0} \setcounter{remark}{0}

Our main objective in this section is to give an alternative proof
for $|{\cal PC}_n|$. We begin our investigation by considering
lattice paths in the Cartesian plane that start at $(0,0)$, end at
$(n,n)$, contain no points above the line $y=x$, and composed only
of steps $(1,0), (0,1)$ and $(1,1)$, i.e., $\rightarrow, \uparrow$
and $\nearrow$. We shall call such lattice paths {\em type I}, in
this paper. The diagrams in Figure 1 illustrate all such paths in
the $1\times 1$ and $2 \times 2$ squares, respectively. The total
number of such paths is known to be the large or double
Schr\"{o}der number $r_n$, \cite{Per}.

\begin{figure}[!htb]
\begin{center}
\epsfig{file=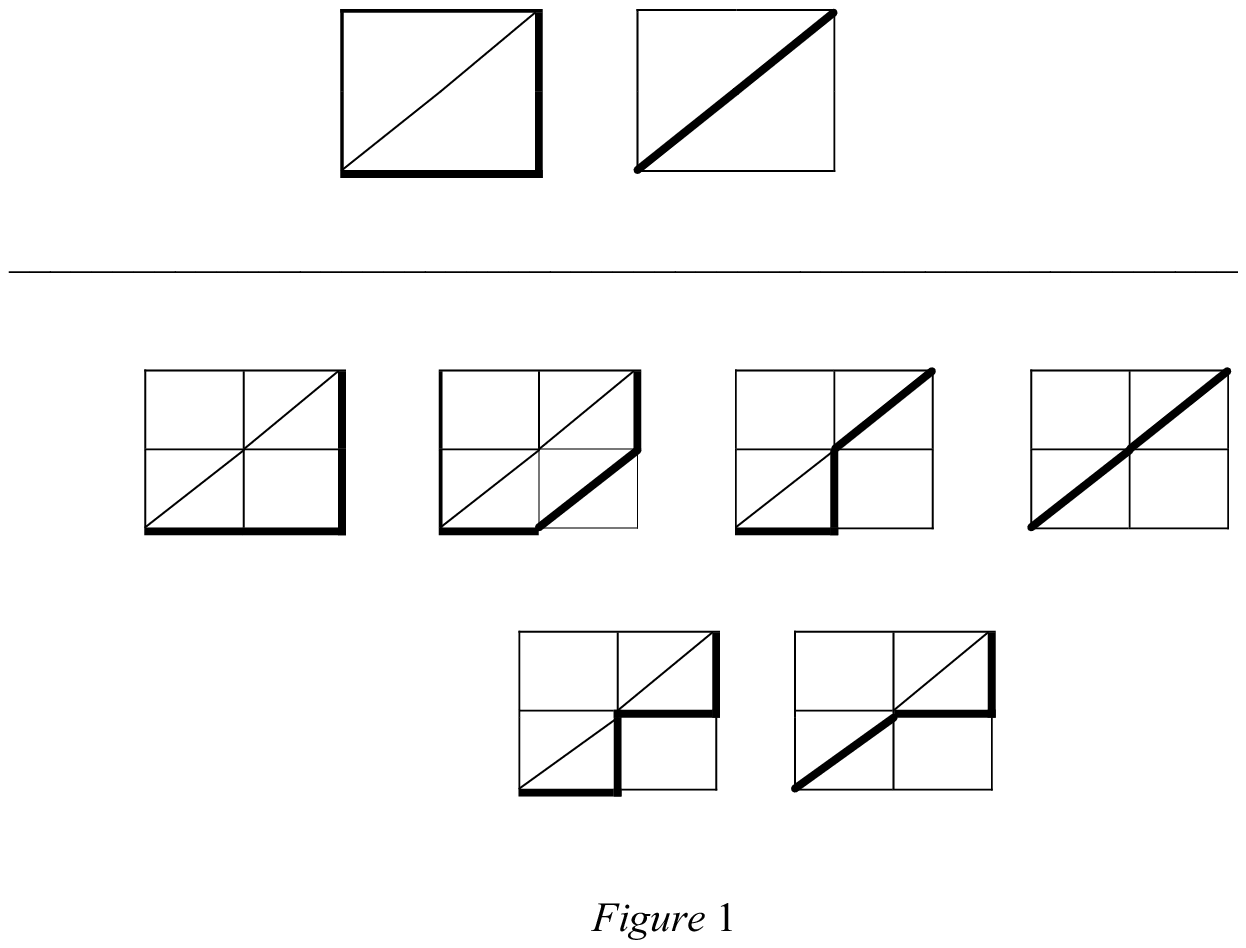}
\end{center}
\end{figure}

To establish a bijection between the set of type I paths (in an
$n\times n$ square) and the set of all decreasing and
order-preserving partial transformations of $X_n$, we first make
the following observation:
\begin{enumerate} \item[(OB1)] In each lattice path from $(0,0)$ to $(n-1,n-1)$
there are exactly same number of horizontal steps as there are
vertical steps.
\end{enumerate}
To see this, we note that a horizontal step corresponds to the map
$h: (a,b)\rightarrow (a+1,b)$, a vertical step corresponds to the
map $v: (a,b)\rightarrow (a,b+1)$, and a diagonal step corresponds
to the map $d: (a,b)\rightarrow (a+1,b+1)$. It is easy to check
that $hv=vh, hd=dh$ and $vd=dv$. Now if a path starts from $(0,0)$
and ends at $(n,n)$ using only the above steps then each such path
can be written as a product: $h_1h_2 \cdots h_rv_1v_2 \cdots
v_sd_1d_2 \cdots d_t$, and so we have
\[
(n,n)=h_1h_2 \cdots h_rv_1v_2 \cdots v_sd_1d_2 \cdots
d_t(0,0)=(t+r,t+s).\] Hence $t+r=n=t+s$, which implies $r=s$.

Next, we note that it is convenient (in this section) to express
$\alpha$ in ${\cal P}{\cal C}_n$ (with base set $X_n =
\{0,1,2,\ldots, n-1\}$) as
\begin{equation}
\alpha = \left(\begin{array}{cccc}
a_1 & a_2 \cdots a_r\\
b_1 & b_2 \cdots b_r \end{array} \right)
\end{equation}
where the $a_i$'s are distinct, but the $b_i$'s are not
necessarily distinct.  Moreover, we may also assume that $0 \leq
a_1 < a_2 \cdots < a_r \leq n-1$ and $0 \leq b_1 \leq b_2 \leq
\cdots \leq b_r \leq n-1$. Thus
\[
\mbox{Dom }\alpha =\{a_1, a_2, \ldots a_r\},\] however, we refer
to the sequence $(b_1, b_2, \ldots, b_r)$ as the {\em simage} of
$\alpha$, denoted by
\[
\mbox{Sim } \alpha = (b_1, b_2, \ldots, b_r).\] (Note that if all
the $b_i$'s are distinct then $\mbox{Sim } \alpha$ coincides with
$\mbox{Im } \alpha$, otherwise $\mbox{Sim } \alpha \neq \mbox{Im
}\alpha $ ) Now to each vertical step from $(i,j)$ to $(i, j+1)$
in an arbitrary type I path (in an $n\times n$ square) we put $j$
in $\mbox{Dom } \alpha$, and to each horizontal step from $(i,j)$
to $(i+1,j)$ we put $j$ in $\mbox{Sim }\alpha$. The domain is then
arranged in a strictly increasing order while the simage is
arranged in a nondecreasing order, and by virtue of (OB1) this
gives rise to a unique order-preserving map. Two examples should
make these ideas more clear.  The path given in Figure 2 implies
$\mbox{ Dom } \alpha = \{1,3\}$ and $\mbox{Sim }\alpha = (1,1)$.
Thus the associated order-preserving map is
\[
\alpha = \left(\begin{array}{cc} 1 & 3 \\ 1 & 1 \end{array}
\right) \in {\cal P}{\cal C}_4, \] and the path given in Figure 3
implies $\mbox{Dom } \beta = \{0, 2, 3\}$ and $\mbox{Sim }\beta =
(0,2,2)$. Thus the associated order-preserving map is
\[
\beta = \left(\begin{array}{ccc} 0 & 2 & 3 \\
0 & 2 & 2 \end{array} \right)\in {\cal P}{\cal C}_4.\]

\begin{figure}[!htb]
\begin{center}
\epsfig{file=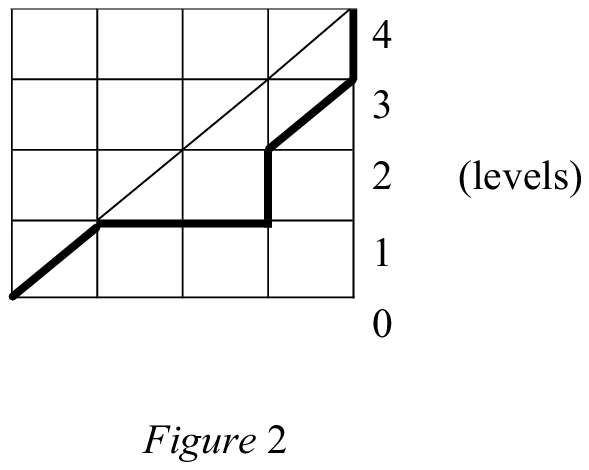}
\end{center}
\end{figure}

\begin{figure}[!htb]
\begin{center}
\epsfig{file=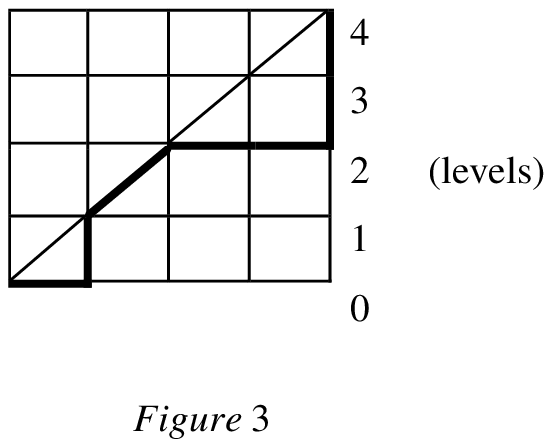}
\end{center}
\end{figure}

\noindent Moreover, the condition that type I paths never cross
above the diagonal line $y=x$, forces these associated
order-preserving maps to be decreasing as well.  In fact, at any
level, a horizontal step will take us 1-unit away from the
diagonal line $y=x$, and a vertical step will take us 1-unit
closer to the diagonal line $y=x$. Thus, since our target is the
point $(n-1,n-1)$ on this diagonal line we must take vertical
steps at higher levels to compensate for each horizontal step, and
so a decreasing map results.  Finally, note that the unique all
diagonal path corresponds to the empty map. Hence we have shown
that to every type I lattice path in an $n\times n$ square there
corresponds a unique partial order-preserving map.

Conversely, we show that every $\alpha \in{\cal P}{\cal C}_n$
corresponds to a type I path in an $n\times n$ square. For a given
$\alpha$ in ${\cal P}{\cal C}_n$ we first express it as in (3.1)
and write
\[
 \mbox{Dom } \alpha = \{a_1, a_2,\ldots, a_r\},\; \mbox{Sim }\alpha
= (b_1, b_2, \ldots, b_r) \] for some $r$ in $\{0,1,\ldots,n-1\}$.
To construct an associated type I path, first note that since, in
general, there may be repetitions in $\mbox{Sim }\alpha$, we shall
consider it as consisting of blocks of subsequences, where an
$x$-block consists only of $x$ in $\{0, 1,2,\ldots,n-1\}$,
repeated say $m$ times $(1 \leq m \leq n)$. Thus we may write
\[
\mbox{Sim } \alpha = (x_1\mbox{-block}; x_2\mbox{-block}; \ldots;
x_s\mbox{-block})
\]
where $0 \leq x_1 < x_2 < \cdots < x_s \leq n-1$. (Note that
$\mbox{Im }\alpha = \{x_1, x_2, \ldots, x_s\}.)$ Next we label the
horizontal rows of the $n \times n$ square from the bottom to the
top, starting from 0 to $n-1$. Now starting from $(0, 0),$ if $x_1
= 0$ take as many horizontal steps as the length of the
$x_1$-block, otherwise take a diagonal step to the next level
(level 1). Note that the order-decreasing property guarantees that
$b_1 = x_1 \leq a_1$, however there is no guarantee that $x_i \leq
a_i$ (for $i > 1)$, but this will not be of any disadvantage.  In
general, at level $m$\, $(0 \leq m < n-1)$, check
\begin{enumerate}
\item[(i)] if $m\in \mbox{Im }\alpha$ and $m\in \mbox{Dom
}\alpha$, take as many horizontal steps as the length of the
$m$-block followed by a vertical step to the next level;
\item[(ii)] if $m\in \mbox{Im }\alpha$ and $m \not \in \mbox{Dom
}\alpha$, take as many horizontal steps as the length of the
$m$-block followed by a diagonal step to the next level;
\item[(iii)] if $m \not \in \mbox{Im }\alpha$, take a vertical or
diagonal step to the next level depending on whether $m\in
\mbox{Dom }\alpha$ or $m\not\in \mbox{Dom }\alpha$, respectively.
\end{enumerate}

Now since in (i) and (ii) we give priority to horizontal step(s)
over vertical and diagonal steps, it follows by the
order-decreasing property that at level $m$ we must have had at
least as many horizontal steps as there are vertical steps (up to
level $m+1)$. This in turn guarantees that our paths never
overshoot to cross the diagonal line $y=x$.  Moreover the
bijection between $\mbox{Dom }\alpha$ and $\mbox{Sim }\alpha$
guarantees that our paths always end up at $(n-1,n-1)$ as
required.

An example would be quite appropriate.  Consider the map
\[
\alpha = \left(\begin{array}{ccccc}
0 & 2 & 3 & 5 & 6 \\
0 & 0 & 0 & 4 & 4 \end{array} \right)\in {\cal PC}_7.\] Then
\begin{eqnarray*}
\mbox{Dom }\alpha  =  \{0,2,3,5,6\},\,\, \mbox{Sim } \alpha & = &
(0,0,0;4,4), \quad \mbox{Im }\alpha = \{0,4\}.
\end{eqnarray*}
Now start from $(0,0)$ and take 3 horizontal steps, since the
first block of $\mbox{Sim }\alpha$ is a 0-block of length 3. We
are still at level 0, however, we take a vertical step to level 1
since $0\in \mbox{Dom }\alpha$. Now since $1\not\in (\mbox{Dom
}\alpha)\cap (\mbox{Im } \alpha)$, we take a diagonal step to
level 2.  Next, since $2\in \mbox{Dom }\alpha \cap (\mbox{Im
}\alpha)'$ we take a vertical step to level 3. Then since $3\in
\mbox{Dom }\alpha\cap(\mbox{Im }\alpha)'$ we take a vertical step
to level 4.  Now we take 2 horizontal steps and then a diagonal
step since $4\in \mbox{Im }\alpha \cap (\mbox{Dom }\alpha)'$. We
are now at level 5, from where we take a vertical step to level 6
since $5\in \mbox{Dom }\alpha \cap(\mbox{Im }\alpha)'$. Finally,
we take another vertical step to level 7, since $6\in \mbox{Dom
}\alpha \cap(\mbox{Im }\alpha)'$. Thus we have the path indicated
in Figure 4.

\begin{figure}[!htb]
\begin{center}
\epsfig{file=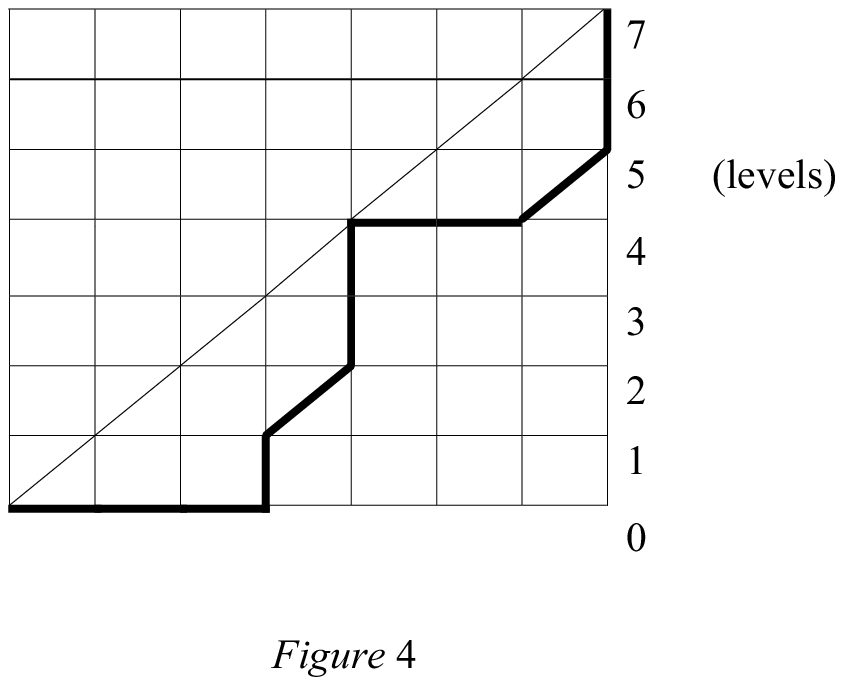}
\end{center}
\end{figure}

We have now established the main result of this section:

\begin{theorem}\label{thrm31}
There is a one-one correspondence between the set of all
decreasing and order-preserving partial transformations of an
$n$-element chain and the set of all type I lattice paths in an
$n\times n$ square.
\end{theorem}

An immediate consequence of Theorem \ref{thrm31} is that it
furnishes us with the order of $|{\cal PC}_n|$. However, before we
formally state this result we first deduce from Pergola and
Sulanke \cite{Per} and Stanley \cite{Sta} that the large (or
double) Schr\"{o}der number denoted by $r_n$ could be defined as
\[
r_n = \frac{1}{n+1} \sum^n_{r=0} \left(\begin{array}{c} n+1 \\ n-r
\end{array} \right) \left(\begin{array}{c} n+r \\ r \end{array}
\right).\] Moreover, $r_n$ satisfies the recurrence
\begin{equation}
(n+2)r_{n+1} = 3(2n+1) r_n - (n-1)r_{n-1}
\end{equation}
for $n\geq 1$, with initial conditions $r_0 = 1$ and $r_1 = 2$.
The (small) Schr\"{o}der number is usually denoted by $s_n$ and
defined as $s_0 = 1, s_n = r_n/2 \, (n \geq 1)$ and so it
satisfies the same recurrence as $r_n$.

\begin{theorem}\cite[Theorem 2.10]{Lar2}\label{thrm32}
Let ${\cal PC}_n$ be as defined in (2.1). Then $|{\cal PC}_n| =
r_n$, the double Schr\"{o}der number.
\end{theorem}

Several consequences of this result will now be exhibited. For the
semigroup $Q_n$ (defined in (2.2)) and its set complement $Q'_n$,
we now have
\begin{corollary}
$|Q_n| = |Q'_n|=s_n$, the (small) Schr\"{o}der number.
\end{corollary}

It is also known (see for example, Pergola and Sulanke\cite{Per})
that the number of type I lattice paths (in an $n\times n$ square)
without a diagonal step
is the $n$-$th$ Catalan number: $\left(\begin{array}{c} 2n\\
n\end{array}\right)/(n+1)$. However, since clearly such paths
correspond to full transformations, Higgins \cite[Theorem
3.1]{Hig2} follows immediately:

\begin{theorem}\label{thrm34}
Let \,${\cal C}_n$ be the semigroup of all decreasing and
order-preserving full transformations of $X_n$. Then $|{\cal
C}_n|= \left(\begin{array}{c} 2n\\ n\end{array} \right)/(n+1)$,
the $n$-$th$ Catalan number.
\end{theorem}

Further, we give a sequence of results which arise naturally in
the semigroup context which (with the exception of Theorem
\ref{thrm38}) may not have been asked yet in lattice path theory.
First, as in \cite{Lar2,Lar1} we define the numbers $F(n, r)$,
$G(n, k)$ and $J(n,r)$, respectively by
\begin{equation}
F(n,r) = |\{\alpha\in S: |\mbox{Dom }\alpha| = r \}|,
\end{equation}
\begin{equation}
G(n,k) = |\{\alpha\in S: \max(\mbox{Im }\alpha) = k\}|,
\end{equation}
\begin{equation}
J(n,r) = |\{\alpha\in S: |\mbox{Im }\alpha| = r \}|,
\end{equation} where $S$ is an arbitrary semigroup of
transformations. Then from (3.3) and \cite[Proposition 2.8]{Lar2}
we deduce
\begin{theorem}\label{thrm35}
Let $F(n, r)$ be as defined in (3.3), where $S={\cal PC}_n$. Then
the number of type I lattice paths (in an $n \times n$ square)
with exactly r vertical steps is
\[ F(n,r) = \frac{1}{n} \left(\begin{array}{c} n \\ r \end{array}
\right)
 \left(\begin{array}{c} n+r\\ n-1 \end{array} \right).\]
\end{theorem}

\noindent From (3.4) and \cite[Proposition 2.7]{Lar2} we deduce
\begin{theorem}\label{thrm36}
Let $G(n, k)$ be as defined in (3.4), where $S={\cal PC}_n$. Then
the number of type I lattice paths (in an $n \times n$ square) in
which the last horizontal segment is at level $k-1$ is $G(n,k)$,
where $G(n,0) = 1, G(n, 1) = 2^n - 1$, $G(n,n) = r_{n-1}$ and for
$2 \leq k < n$,
 \[
G(n,k) = 2G(n-1, k)-G(n-1, k-1) + G(n, k-1).\]

\end{theorem}

\noindent From (3.4) and \cite[Proposition 3.10]{Lar1} we deduce

\begin{theorem}\label{thrm37}
Let $G(n, k)$ be as defined in (3.4), where $S={\cal C}_n$. Then
the number of type I lattice paths without diagonal steps (in an
$n \times n$ square) in which the last horizontal segment is at
level $k-1$ is $G(n,k)$, where for $1 \leq k \leq n$,
 \[ G(n,k) = \frac{n-k+1}{n} \left(\begin{array}{c} {n+k-2} \\ {n-1} \end{array}
\right).\]
\end{theorem}

From (3.5) and \cite[Proposition 3.6]{Lar1} we deduce

\begin{theorem}\label{thrm38}
Let $J(n, r)$ be as defined in (3.5), where $S={\cal C}_n$. Then
the number of type I lattice paths without diagonal steps (in an
$n \times n$ square) with exactly $r$ horizontal segments is
$J(n,r)$, where for $1 \leq r \leq n$,
 \[ J(n,r) = \frac{1}{n-r+1} \left(\begin{array}{c} n \\ r \end{array}
\right)
 \left(\begin{array}{c} n-1\\ r-1 \end{array} \right).\]

\end{theorem}

Define a subsemigroup of ${\cal PO}_n$ as
\begin{equation}
DL_n = \{\alpha\in {\cal PO}_n: \mbox{Dom }\alpha \subseteq
X_{n-1}= \{0,1,\cdots,n-2\}\}.
\end{equation}
Now if we consider the (unrestricted) lattice paths in the sense
that our lattice paths consist of the same steps as the type I
lattice paths but they can now cross the diagonal line many times
if necessary then the total number of such paths is known to be
the central Delannoy number, $D(n,n)$ \cite{Com}, where the
(arbitrary) Delannoy numbers $D(n,k)$ are given by
\begin{equation}
D(n,k) = \sum^n_{r=0} \left(\begin{array}{c} k \\ r \end{array}
\right)
 \left(\begin{array}{c} n+k-r\\ r \end{array} \right),
\end{equation}
and hence
\begin{equation}
D(n,n) = \sum^n_{r=0} \left(\begin{array}{c} n \\ r \end{array}
\right)
 \left(\begin{array}{c} n+r\\ r \end{array} \right) =\,
 _2F_1(-n;n+1;1;-1),
\end{equation}
where $_2F_1(a,b;c;z)$ is a hypergeometric function. Moreover,
$D(n,n)$ satisfies the recurrence
\begin{equation*}
D(n,n) = D(n-1,n)+D(n,n-1)+D(n-1,n-1).
\end{equation*}

The proof that there is a bijection between the set of
(unrestricted) lattice paths in an $n \times n$ square and the
semigroup $DL_n$ is for most part similar to the one given earlier
between the set of restricted (or type I) lattice paths and the
semigroup ${\cal PC}_n$, since $(OB1)$ applies to all lattice
paths. However, note that below the diagonal of the $n \times n$
square a horizontal step takes us 1-unit away from the diagonal
while a vertical step takes us 1-unit closer to the diagonal, but
above the diagonal the reverse is the case. In either case, the
diagonal step is neutral in the sense that it neither takes us
closer to nor away from the diagonal. Moreover, in this case we
need not worry about our lattice paths crossing the diagonal of
the $n \times n$ square, as our paths can now cross over several
times if necessary. As in the above we exemplify how to construct
a lattice path from a partial transformation, and vice-versa. The
path in Figure 5 implies $\mbox{Dom } \beta = \{0,1,3\}$ and
$\mbox{Sim } \beta = \{0;4,4\}$. Thus the associated partial
order-preserving map is
\[
\beta = \left(\begin{array}{ccc} 0 & 1 & 3 \\
0 & 4 & 4 \end{array} \right)\in DL_4.\]

\noindent Similarly, consider the partial order-preserving map

\[
\gamma = \left(\begin{array}{ccccc} 0 & 2 & 3 & 5 & 6 \\
0 & 3 & 3 & 4 & 7 \end{array} \right) \in DL_7.\]

\noindent It is now routine to verify that it corresponds to the
path given in Figure 6.

\begin{figure}[!htb]
\begin{center}
\epsfig{file=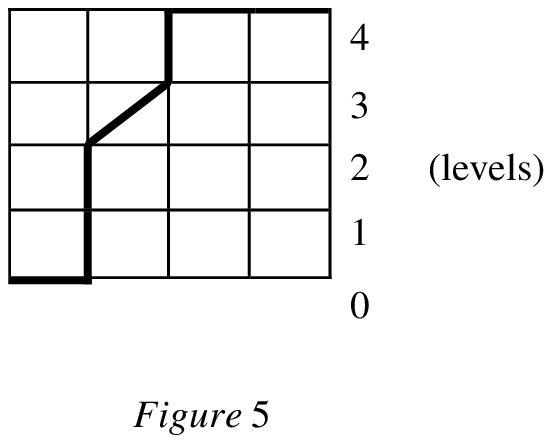}
\end{center}
\end{figure}
\begin{figure}[!htb]
\begin{center}
\epsfig{file=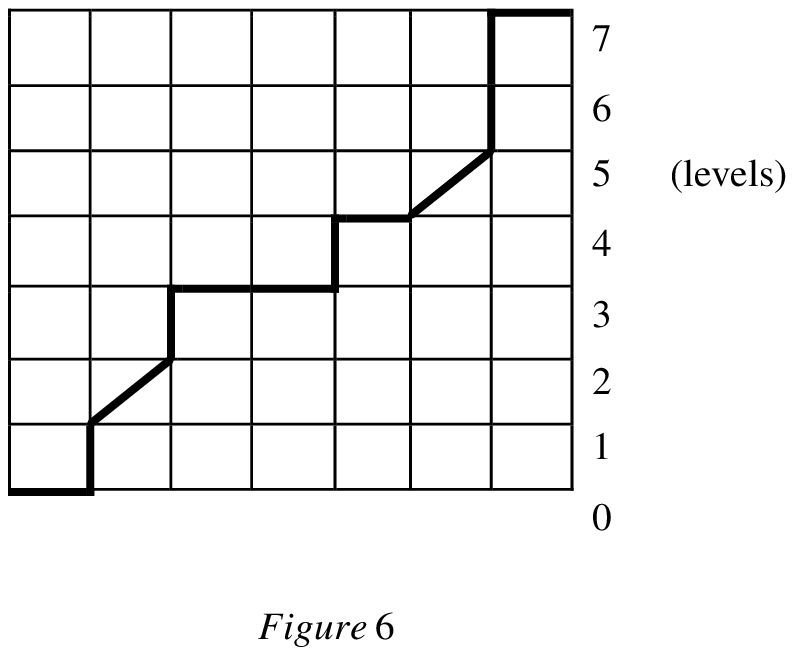}
\end{center}
\end{figure}

\noindent Thus we have established

\begin{theorem}
There is a one-one correspondence between the set of all
order-preserving partial transformations in $DL_n$ and the set of
all (unrestricted) lattice paths in an $n\times n$ square.
\end{theorem}

\noindent Immediately we deduce

\begin{theorem}
Let $DL_n$ be as defined in (3.6). Then $|DL_n| = D(n,n)$, the
n-th central Delannoy number.
\end{theorem}

\noindent From \cite[Proposition 2.11]{Lar3} we deduce

\begin{theorem}
The total number of lattice paths (in an $n \times n$ square)
whose last step is never a horizontal step is $|{\cal
PO}_{n-1}|=c_{n-1}$, where $c_0 = 1, c_1 =2$ and for all $n \geq
1$
\[
(2n-1)(n+1)c_{n+1} = 4(3n^2-1)c_n-(2n+1)(n-1)c_{n-1}.\]

\end{theorem}

\noindent From (3.3) and \cite[Proposition 2.7]{Lar3} we deduce

\begin{theorem}\label{thrm312}
Let $F(n, r)$ be as defined in (3.3), where $S={\cal PO}_n$. Then
the number of lattice paths (in an $(n+1) \times (n+1)$ square)
whose last step is not a horizontal step and with exactly r
vertical steps is
\[ F(n,r) =  \left(\begin{array}{c} n \\ r \end{array}
\right)
 \left(\begin{array}{c} n+r-1\\ n-1 \end{array} \right).\]
\end{theorem}

\noindent From (3.4) and \cite[Lemma 2.5]{Lar3} we deduce

\begin{theorem}\label{thrm313}
Let $G(n, k)$ be as defined in (3.4), where $S={\cal PO}_n$. Then
the number of lattice paths (in an $(n+1) \times (n+1)$ square)
without diagonal steps and in which last step is not a horizontal
step while the last horizontal segment is at level $k-1$ is
$G(n,k)$, where $G(n,0) = 1, G(n, 1) = 2^n - 1$, $G(n,n) =
nr_{n-1} $ and for $2 \leq k < n$,
 \[
G(n,k) = 2G(n-1, k)-G(n-1, k-1) + G(n, k-1).\]

\end{theorem}

\noindent From (3.5) and \cite[Lemma 4.5]{Lar3} we deduce

\begin{theorem}\label{thrm314}
Let $J(n, r)$ be as defined in (3.5), where $S={\cal PO}_n$. Then
the number of lattice paths without diagonal steps (in an $(n+1)
\times (n+1)$ square), whose last step is not a horizontal step
and with exactly $r$ horizontal segments is $J(n,r)$, where
$J(n,0)=1=J(n,n)$, and for $1 \leq r < n$,
 \[ \l(\ba{c} n-1 \\ r-1 \ea \r)J(n,r) = \fr{2(n-r+1)}{n-r}\l(\ba{c} n \\ r-1 \ea
 \r)J(n-1,r)
  + \l(\ba{c} n \\ r \ea \r) J(n-1,r-1)\]

\end{theorem}

\noindent From (3.4) and \cite[Corollary 3.11]{Lar5} we deduce

\begin{theorem}\label{thrm315}
Let $G(n, k)$ be as defined in (3.4), where $S={\cal O}_n$. Then
the number of lattice paths (in an $n \times n$ square) without
diagonal steps and in which the last horizontal segment is at
level $k-1$ is $G(n,k)$, where  for $1 \leq k \leq n$,
 \[
G(n,k) = \l( \ba{c} n +k -2 \\ k-1 \ea \r).\]

\end{theorem}

\noindent From (3.5) and \cite[Corollary 3.10]{Lar5} we deduce

\begin{theorem}\label{thrm316}
Let $J(n, r)$ be as defined in (3.5), where $S={\cal O}_n$. Then
the number of lattice paths without diagonal steps (in an $n
\times n$ square) with exactly $r$ horizontal segments is
$J(n,r)$, where for $1 \leq r \leq n$,
 \[ J(n,r) = \l(\ba{c} n \\ r \ea \r) \l(\ba{c} n-1 \\ r-1 \ea \r)\]

\end{theorem}

\section{Concluding Remarks}
\setcounter{equation}{0} \setcounter{lemma}{0}
\setcounter{theorem}{0} \setcounter{corollary}{0}
\setcounter{proposition}{0} \setcounter{remark}{0}

\begin{remark} {\rm As far as we know the subsemigroups $Q_n$, $Q'_n$
and $DL_n$ have not been studied.}
\end{remark}

\begin{remark} {\rm We can now compose lattice paths via the
bijection with certain subsemigroups of ${\cal PO}_n$, thereby
making these lattice paths into algebraic objects. A more
important consequence is we believe, the link of combinatorial
questions between lattice paths and partial transformations is now
firmly established.}
\end{remark}

As stated in the introduction the number of idempotents of various
classes of semigroups of transformations has been computed.  For
further results see \cite{How3,Lar3,Umar1,Umar2}. The number of
idempotents $|E({\cal PC}_n)|$, has been shown in \cite{Lar2} to
be $(3^n+1)/2$.

\begin{remark} {\rm
It is not very difficult to see that lattice paths in which a
diagonal step never succeeds a horizontal segment, and where every
length $k$ horizontal segment is followed by exactly $k$ vertical
steps plus some (may be none) diagonal steps before another
horizontal segment, correspond to idempotents in ${\cal PO}_n$.
However, while idempotents are natural elements to study in a
semigroup, their corresponding paths do not seem to have a natural
description.}
\end{remark}

We would like to take this opportunity to give an alternative
easier proof of the formula for $|E({\cal PC}_n)|$, though quite
admittedly not via lattice paths this time. First, we consider
\begin{equation*}
e(n,r) = |\{\alpha \in {\cal PC}_n: \alpha^2 = \alpha, |\mbox{Im
}\alpha| = r \}|.
\end{equation*}
Then clearly we have
\[
e(n,0) = 1 = e(n,n),\] where the former corresponds to the empty
map and the latter corresponds to the identity map. More
generally, we have

\begin{lemma}\label{4.1} For all  $n \geq r\geq 1$, we have
\[ e(n,r) = 2e(n-1, r) + e(n-1, r-1).\]
\end{lemma}
\begin{proof}{}
If $n\not\in \mbox{Dom }\alpha$ then $n\not\in \mbox{Im } \alpha$,
by idempotency and so there are $e(n-1, r)$ idempotents of this
type.  If on the other hand $n \in \mbox{Dom } \alpha$ then either
$n\alpha = k\alpha < n$, for some $k \in \{r,r+1, \cdots, n-1\}$,
of which there are again $e(n-1, r)$ idempotents of this type; or
$n\alpha = n$, of which there are $e(n-1,r-1)$ idempotents of this
type, by the order-decreasing property. Hence the result follows.
\end{proof}

Now let $e_n = \sum^n_{r=0} e(n,r)$. Then $e_0 = 1$, and the next
lemma gives a recurrence relation satisfied by $e_n$.

\begin{lemma} For all  $n \geq 1$, we have:
$e_n = 3e_{n-1}-1.$
\end{lemma}
\begin{proof}{} By using Lemma \ref{4.1}, we have
\begin{eqnarray*}
e_n & = & \sum^n_{r=0} e(n,r)
 = e(n,0)+e(n,1)+e(n,2)+ e(n,3)+ ... +e(n-1,n-1)+e(n,n)\\
&=&  [2e(n-1,0) - 1]+[2e(n-1,1) +e(n-1,0)]+[2e(n-1,2) +e(n-1,1)] \\
 && + [2e(n-1,3) +e(n-1,2)]+ ... +[2e(n-1,n) +e(n-1,n-1)] \\
& = & 3e(n-1,0)+3e(n-1,1)+3e(n-1,2)+ ... +3e(n-1,n-1)-1 \\
 & = & 3\sum^{n-1}_{r=0}e(n-1,r) -1 =3e_{n-1}-1.
 \end{eqnarray*}
 \end{proof}

 We now have \cite[Proposition 3.5]{Lar2}.

\begin{theorem} Let ${\cal P}{\cal C}_n$ be as defined in (2.1).
Then $|E({\cal P}{\cal C}_n)| = e_n = \frac{1}{2}(3^n + 1)$.
\end{theorem}
\begin{proof}{}
By the standard method of solving linear recurrence relations. See
Anderson \cite{And}, for example.
\end{proof}
\vspace{4mm}

\noindent {\bf Acknowledgment}. We would like to gratefully
acknowledge support from  King Fahd University of Petroleum \&
Minerals, and Sultan Qaboos University.


\baselineskip18pt

\begin{thebibliography}{99}

\bibitem{And}
I. Anderson, {\em A first course in combinatorial mathematics},
(Oxford University Press, 1974).

\bibitem{Cli}
A. H. Clifford and G. B. Preston, {\em The algebraic theory of
semigroups}, Vol. 1, Mathematical Surveys 7 (Providence, R. I.:
American Math. Soc., 1961).

\bibitem{Com}
L. Comtet, {\em Advanced Combinatorics: the art of finite and
infinite expansions}, D. Reidel Publishing Company, Dordrecht,
Holland: 1974.

\bibitem{Gar1}
G.U. Garba, Idempotents in partial transformation semigroups. {\em
Proc. Roy. Soc. Edinburgh} {\bf 116} (1990), 359--366.

\bibitem{Gar2}
G. U. Garba, On the idempotent ranks of certain semigroups of
order-preserving transformations,  {\em Portugaliae Mathematica}
{\bf 51} (1994), 185--204.

\bibitem{Gom}
G. M. S. Gomes and J. M. Howie, On the ranks of certain semigroups
of order-preserving transformations, \emph{Semigroup Forum},
\textbf{45} (1992), 272--282.

\bibitem{Hig1}
P. M. Higgins, {\em Techniques of semigroup theory,} (Oxford
University Press, 1992).

\bibitem{Hig2}
P. M. Higgins, Combinatorial results for semigroups of
order-preserving mappings, {\em Math. Proc. Camb. Phil. Soc.} {\bf
113} (1993), 281--296.

\bibitem{How1}
J. M. Howie, {\em Fundamentals of semigroup theory} (Oxford:
Clarendon Press, 1995).

\bibitem{How2}
J. M. Howie, Products of idempotents in certain semigroups of
order-preserving transformations, {\em Proc. Edinburgh Math. Soc.}
{\bf 17} (1971), 223--236.

\bibitem{How3}
J. M. Howie, Combinatorial and probabilistic results in
transformation semigroups, {\em Words, Languages and Combinatorics
II} World Sci. Publishing, River Edge, NJ (1994), 200--206.

\bibitem{Lar1}
A. Laradji and A. Umar, On certain finite semigroups of
order-decreasing transformations I, {\em Semigroup Forum} {\bf 69}
(2004), 184--200.

\bibitem{Lar2}
A. Laradji and A. Umar, Combinatorial results for semigroups of
order-decreasing partial transformations, {\em Journal of Integer
Sequences} {\bf 7} (2004), 04.3.8.

\bibitem{Lar3}
A. Laradji and A. Umar, Combinatorial results for semigroups of
order-preserving partial transformations, {\em Journal of Algebra}
{\bf 278} (2004), 342--359.

\bibitem{Lar4}
A. Laradji and A. Umar, On the number of nilpotents in the partial
symmetric semigroup, {\em Communications in Algebra} {\bf 32}
(2004), 3017--3023.

\bibitem{Lar5}
A. Laradji and A. Umar, Combinatorial results for semigroups of
order-preserving full transformations, {\em Semigroup Forum} {\bf
72} (2006), 51--62.

\bibitem{Per}
E. Pergola and R. A. Sulanke, Schr\"{o}der Triangles, Paths and
Parallelogram Polyominoes, {\em Journal of Integer Sequences} Vol.
{\bf 1} (1998), 98.1.7.

\bibitem{Slo}
N. J. A. Sloane, The On-Line Encyclopedia of Integer Sequences,
available at \\http://www.research.att.com/~njas/sequences/.

\bibitem{Sta}
R. P. Stanley, Hipparchus, Plutarch, Schr\"{o}der and Hough, {\em
Amer. Math. Monthly}  {\bf 104} (1997), 344--350.

\bibitem{Tai}
M. Tainiter, A characterisation of idempotents in semigroups, {\em
J. Combin.} {\em Theory} {\bf 5} (1968), 370--373.

\bibitem{Umar1}
A. Umar, On the semigroups of order-decreasing finite full
transformations, {\em Proc. Roy. Soc. Edinburgh Sect.} A {\bf 120}
(1992), 129--142.

\bibitem{Umar2}
A. Umar, Enumeration of certain finite semigroups of
transformations, {\em Discrete Math.} {\bf 189} (1998), 291--297.

\bibitem{Umar3}
A. Umar, On certain infinite semigroups of order-decreasing
transformations I, {\em Communications in Algebra} Vol. {\bf 25}
(9) (1997), 2987--2999.

\bibitem{Vag}
V.V. Vagner, Representations of ordered semigroups. Math. Sb. NS.
{\bf 387} (1956), 203-240, translated in Amer. Math. Soc. Trans.
(2) {\bf 36} (1964), 295--336.

\end{thebibliography}
\end{document}